\documentclass[12pt,
psamsfonts]{amsart}

\usepackage{amssymb,amsfonts,amsmath}
\usepackage[all,arc]{xy}
\usepackage{enumerate}
\usepackage{mathrsfs}
\usepackage{fullpage}
\usepackage{xspace}
\usepackage[margin=1.0in]{geometry}
\usepackage{tcolorbox}
\usepackage{tikz-cd}
\usepackage{color}
\usepackage{aliascnt}
\usepackage[foot]{amsaddr}
\usepackage{hyperref}

\newtheorem{thm}{Theorem}[section]

\newaliascnt{theo}{thm}
\newtheorem{theo}[theo]{Theorem}
\aliascntresetthe{theo}

\newaliascnt{cor}{thm}
\newtheorem{cor}[cor]{Corollary}
\aliascntresetthe{cor}

\newaliascnt{prop}{thm}
\newtheorem{prop}[prop]{Proposition}
\aliascntresetthe{prop}

\newaliascnt{lem}{thm}
\newtheorem{lem}[lem]{Lemma}
\aliascntresetthe{lem}

\newaliascnt{conj}{thm}

\aliascntresetthe{conj}

\newaliascnt{que}{thm}

\aliascntresetthe{que}

\newaliascnt{ass}{thm}

\aliascntresetthe{ass}

\newaliascnt{defn}{thm}
\newtheorem{defn}[defn]{Definition}
\aliascntresetthe{defn}

\theoremstyle{remark}
\newaliascnt{rem}{thm}
\newtheorem{rem}[rem]{Remark}
\aliascntresetthe{rem}

\theoremstyle{definition}

\newtheorem{exmp}[thm]{Example}

\newtheorem{notn}[thm]{Notation}

\newcommand{\Z}{\mathbb{Z}\xspace}
\newcommand{\Q}{\mathbb{Q}\xspace}
\newcommand{\G}{\mathbb{G}\xspace}
\DeclareMathOperator{\Spec}{Spec}
\DeclareMathOperator{\res}{res}

\DeclareMathOperator{\alb}{alb}
\DeclareMathOperator{\img}{Im}

\DeclareMathOperator{\Hom}{Hom}

\DeclareMathOperator{\Gal}{Gal}

\DeclareMathOperator{\Alb}{Alb}
\DeclareMathOperator{\Ext}{Ext}
\DeclareMathOperator{\ab}{ab}

\makeatletter
\let\c@equation\c@theo
\makeatother
\numberwithin{equation}{section}

\title{A Tate duality theorem for local Galois symbols II; \\ The semi-abelian case}

\author[*]{Evangelia Gazaki*} \address[*]{\normalfont Department of Mathematics, University of Michigan, 3823 East Hall, 530 Church St., Ann Arbor, MI, 48109, USA. Email: \texttt{gazaki@umich.edu}}

\begin{document}

\maketitle

\begin{abstract} This paper is a continuation to \cite{Gazaki2017}. For every integer $n\geq 1$, we consider the generalized Galois symbol $K(k;G_1,G_2)/n\xrightarrow{s_n} H^2(k,G_1[n]\otimes G_2[n])$, where $k$ is a finite extension of $\Q_p$, $G_1,G_2$ are semi-abelian varieties over $k$ and $K(k;G_1,G_2)$ is the Somekawa K-group attached to $G_1, G_2$. Under some mild assumptions, we describe the exact annihilator of the image of $s_n$ under the Tate duality perfect pairing, $H^2(k,G_1[n]\otimes G_2[n])\times H^0(k,\Hom(G_1[n]\otimes G_2[n],\mu_n))\rightarrow\Z/n$. An important special case is when both $G_1, G_2$ are abelian varieties  with split semistable reduction. In this case we prove a finiteness result, which gives an application to zero-cycles on abelian varieties and products of curves. 

\end{abstract}

\section{Introduction}
In the early 90's Somekawa (\cite{Somekawa1990}), following a suggestion of K. Kato, defined a new K-group, $K(k;G_1,\cdots,G_r)$ attached to a finite family of semi-abelian varieties over a field $k$. This group is a generalization of the Milnor K-group, $K_r^M(k)$, of a field $k$. In fact, in the most basic case when $G_1=\cdots=G_r=\G_m$, Somekawa proved an isomorphism \[K_r^M(k)\simeq K(k;\G_m,\cdots,\G_m).\]  Moreover, similarly to the case of the classical Milnor K-group, for every integer $n\geq 1$ invertible in $k$, there is a well-defined homomorphism to Galois cohomology,
\[s_n:K(k;G_1,\cdots,G_r)/n\rightarrow H^r(k,G_1[n]\otimes\cdots\otimes G_r[n]),\] known as the \textit{generalized Galois symbol}. The map $s_n$ is essentially constructed from the connecting homomorphisms, $G_i(k)/n\hookrightarrow H^1(k,G_i[n])$, by taking cup products.  

The goal of this article is to study the image of this map, $s_n$, when the base field $k$ is a finite extension of the $p$-adic field $\Q_p$ and when $r=2$. Note that because every such field has cohomological dimension two, the only interested cases are when $r=1,2$. 

The study of Galois symbols is classical and dates back to the work of Tate (\cite{Tate1976}). Obtaining information on either the image or the kernel of these maps is a problem of very high difficulty. We indicatively mention the Bloch-Kato conjecture (\cite{Bloch/Kato1986}), which was  proved only very recently by Voevodsky (\cite{Voevodsky2011}). The reason people study these Galois symbols is because they carry deep arithmetic information. For example, in specific cases they can be realized as reciprocity maps of geometric class field theory or as cycle maps from algebraic cycles to \'{e}tale cohohomology. 
 
 When the base field $k$ is a finite extension of $\Q_p$ the results concerning the image of $s_n$ vary a lot depending on the coordinates of the Somekawa K-group, $K(k;G_1,\cdots,G_r)$. Here are some significant examples.
 \begin{exmp}\label{example1} The classical Galois symbol, $K_2^M(k)/n\xrightarrow{s_n} H^2(k,\mu_n^{\otimes 2})$, is always an isomorphism for every base field $k$. This is known as the Merkurjev-Suslin theorem (\cite{Merkurjev/Suslin1982}), but when $k$ is $p$-adic this result is due to Tate (\cite{Tate1976}).  
 \end{exmp}
 \begin{exmp}\label{example2} Bloch (\cite{Bloch1981}) studied the Galois symbol,
 \[K(k;J(C),\G_m)/n\xrightarrow{s_n}H^2(k,J(C)\otimes\mu_n),\] where $J(C)$ is the Jacobian of a smooth projective curve over a $p$-adic field $k$. The study of this map was essential to class field theory of curves. When $C$ has good reduction, Bloch showed that $s_n$ is surjective, a result which is no longer true when $C$ has bad reduction (\cite{Saito1985}).  
 \end{exmp}
 \begin{exmp}\label{example3} When $A,B$ are abelian varieties, the local Galois symbol 
\[K(k;A,B)/n\xrightarrow{s_n}H^2(k,A[n]\otimes B[n])\] has been studied extensively by numerous authors (\cite{Raskind/Spiess2000}, \cite{Yamazaki2005}, \cite{Murre/Ramakrishnan2009}, \cite{hiranouchi/hirayama}.  
In this case, it appeared hard to give a definitive answer regarding the image of the Galois symbol, since all results depended heavily on the specific reduction type of $A$, $B$ as well as on the ramification of the field $k$. 
 
In \cite{Gazaki2017} we introduced a new approach to this study,  the use of local Tate duality. Namely, we proved the following theorem.
\begin{theo}\label{goodred} (\cite[Theorem 1.1]{Gazaki2017}) Let $k$ be a finite extension of the $p$-adic field $\Q_p$ with $p\geq 5$ and let $G_k$ be the Galois group of $k$. Let $A,B$ be abelian varieties over $k$ with good reduction. Let $B^\star$ be the dual abelian variety of $B$ and $\mathcal{A},\mathcal{B}^\star$ the N\'{e}ron models of $A,B^\star$ respectively. For a positive integer $n\geq 1$, we consider the Tate duality perfect pairing,
\[\langle ,\rangle:H^2(k,A[n]\otimes B[n])\times\Hom_{G_k}(A[n],B^\star[n])\rightarrow\Z/n.\] The exact annihilator under Tate duality of the image, \[\img(K(k;A,B)\xrightarrow{s_n} H^2(k,A[n]\otimes B[n]))\] of the Galois symbol  consists of those homomorphisms $g:A[n]\rightarrow B^\star[n]$ that lift to a homomorphism $\tilde{g}:\mathcal{A}[n]\rightarrow\mathcal{B}^\star[n]$ of the corresponding finite flat group schemes over $\Spec(\mathcal{O}_k)$. When the integer $n$ is coprime to $p$, the assumption $p\geq 5$ can be dropped. 
\end{theo} 
 \end{exmp}


K. Kato visioned that a more general Tate duality theorem should exist about the image of the Galois symbol 
\[K(k;G_1,G_2)/n\xrightarrow{s_n}H^2(k,G_1[n]\otimes G_2[n])\] for $G_1, G_2$ general semi-abelian varieties. Such a theorem should potentially explain all the different phenomena that appeared in \autoref{example1}, \autoref{example2}, \autoref{example3} and at the same time take also into consideration cases of \textit{bad reduction}. 
This is precisely the goal of the present article. 
\subsection{A filtration on the torsion points}\label{conditions}
The first step will be to extend \autoref{goodred} to abelian varieties with \textit{split semistable reduction} (see \autoref{semistabledefn}).  In order to state our Tate duality theorem, we need to briefly recall some facts  about this type of reduction. More details will be given in \autoref{background} of the paper.

When the abelian variety $A$ has split semistable reduction, the group scheme $\mathcal{A}[n]$ is in general not finite over $\Spec(\mathcal{O}_k)$. In turn, it has a filtration,
\begin{eqnarray}\label{fil}\mathcal{A}[n]\supset\mathcal{A}[n]^f\supset
\mathcal{A}[n]^t\supset 0,\end{eqnarray} with the \textit{finite piece}, $\mathcal{A}[n]^f$, being defined as the maximal finite subgroup scheme of $\mathcal{A}[n]$. The subgroup $\mathcal{A}[n]^t$ is called the \textit{toric piece} of $\mathcal{A}[n]$ and it is constructed using the maximal formal sub-torus of the formal group of $A$. This filtration has the following properties. 
\begin{itemize}
\item $\mathcal{A}[n]^t\simeq\mu_n^{\oplus r}$ for some $r\geq 0$,
\item $\mathcal{A}[n]^f/\mathcal{A}[n]^t\simeq\mathcal{C}[n]$, where $\mathcal{C}$ is an abelian scheme over $\Spec(\mathcal{O}_k)$. 
We will denote this quotient by $\mathcal{A}[n]^{\ab}$ and will call it \textit{the abelian piece} of $\mathcal{A}[n]$. 
\end{itemize} This filtration is classical and has been studied in \cite{SGA7(I)} and \cite{Faltings/Chai1990}. Furthermore, it induces a filtration of the $G_k$-module $A[n]$ by taking
\[A[n]^\bullet=\mathcal{A}[n]^\bullet(\overline{k}),\;\;\;\text{where }\bullet=f, t.\] 

 In order to extend our theorem to semi-abelian varieties, in \autoref{background} we construct a \textit{finite-toric}  filtration, $G[n]\supset G[n]^f\supset G[n]^t\supset 0$, with similar properties for the $G_k$-module $G[n]$, where $G$ is a semi-abelian variety over $k$ fitting into a short exact sequence 
\[0\rightarrow T\rightarrow G\rightarrow A\rightarrow 0,\] with $T$ a split torus and $A$ an abelian variety over $k$ with split semistable reduction. The quotient $G[n]^f/G[n]^t$  has similar properties to $A[n]^f/A[n]^t$, namely it is isomorphic to $D[n]$, where $D$ is the generic fiber of an abelian scheme $\mathcal{D}$  over $\Spec(\mathcal{O}_k)$. A similar filtration can be defined for the Cartier dual,   $\Hom(G[n],\mu_n)$.
\subsection{Statement of the main Theorem}
The main theorem of this article can be summarized as follows. 
\begin{theo}\label{mainsemiab} Let $k$ be a finite extension of the $p$-adic field $\Q_p$ with $p\geq 5$ and let $G_k$ be the Galois group of $k$. Let $G_1, G_2$ be semi-abelian varieties over $k$ such that for $i=1,2$ we have a short exact sequence \[0\rightarrow T_i\rightarrow G_i\rightarrow A_i\rightarrow 0,\] with $T_i$ a split torus and $A_i$ an abelian variety over $k$ with split semistable reduction. For a positive integer $n\geq 1$, we consider the generalized Galois symbol
\[s_n:K(k;G_1,G_2)/n\rightarrow H^2(k,G_1[n]\otimes G_2[n]).\] 
The orthogonal complement of the image, $\img(s_n)$ under the Tate duality pairing,
\[\langle ,\rangle:H^2(k,G_1[n]\otimes G_2[n])\times\Hom_{G_k}(G_1[n],\Hom(G_2[n],\mu_n))\rightarrow\Z/n,\] consists of those homomorphisms $g:G_1[n]\rightarrow\Hom(G_2[n],\mu_n)$ that satisfy the properties:
\begin{itemize}
\item $g$ preserves the aforementioned filtration.
\item The induced morphism $g: G_1[n]^{f}/G_1[n]^t\rightarrow \Hom(G_2[n],\mu_n)^{f}/\Hom(G_2[n],\mu_n)^{f}$ lifts to a homomorphism $\tilde{g}$ of the corresponding finite flat group schemes over $\Spec(\mathcal{O}_k)$. 
\end{itemize} When the integer $n$ is coprime to $p$, the assumption $p\geq 5$ can be dropped. 
\end{theo}
The proof of \autoref{mainsemiab} will be presented in \autoref{proofs} and it will occupy most of this article. In \autoref{abeliancase} we prove the important special case of two abelian varieties $A,B$ with split semistable reduction (\autoref{mainab2}). We note that in certain special cases the assumptions of \autoref{mainsemiab} can be weakened (see for example \autoref{special1}). 


\subsection{An application to zero-cycles} In \autoref{apps} we prove a finiteness result (\autoref{finiteness}), which gives us the following important corollary about zero-cycles. 
\begin{cor}\label{zerocycles} Let $X$ be a smooth projective variety over a finite extension $k$ of $\Q_p$ with $p\geq 5$. We assume that $X$ is  either an abelian variety with split semistable reduction, or a product $C_1\times\cdots\times C_d$ of smooth curves such that for each $i\in\{1,\cdots,d\}$, the curve $C_i$ has a $k$-rational point and the Jacobian variety $J(C_i)$ has split semistable reduction. Then for every prime $l$, the cycle map to \'{e}tale cohomology \[CH_0(X)\rightarrow H^{2d}(X_{et},\Z_l(d)),\] when restricted to the Albanese kernel, $T(X)$, of $X$ has finite image. 
\end{cor} 
The significance of this corollary lies in the case  $l=p$. In fact, it is a conjecture of Colliot-Th\'{e}l\`{e}ne (\cite{Colliot-Thelene1995}) that the Albanese kernel of any smooth projective variety $X$ over a $p$-adic field is the direct sum of a finite group and a divisible group. Recently a weaker form of this conjecture was established by S. Saito and K. Sato (\cite{Saito/Sato2010}). However, the finiteness of $\varprojlim T(X)/p^n$ is still wildly open. \autoref{zerocycles} is an effort to show that at least the cycle map behaves compatibly with the conjecture. 

\subsection{A remark about the method} We note that \autoref{goodred} was proved using two main tools. First, the relation between the K-group $K(k;A,B)$ and the Chow group of zero-cycles, $CH_0(A\times B)$, obtained in \cite{Gazaki2015} together with a result of S. Saito and K. Sato (\cite{Saito/Sato2014}). Second, we used integral comparison theorems for $p$-adic cohomology theories. We believe that a proof of \autoref{mainsemiab} could potentially be obtained using $\log$ analogues of the latter. However, to our knowledge the theory of $\log$ finite flat group schemes has not been fully developed yet, therefore in the current article we chose an alternative method to prove \autoref{mainsemiab}. Namely, we use the theory of $p$-adic uniformization of abelian varieties as discussed in \cite{SGA7(I)} and \cite{Faltings/Chai1990}, in order to reduce to the good reduction case. 

We also note that \autoref{mainsemiab} does not provide a new proof of the results discussed in examples \eqref{example1} and \eqref{example2}. It only puts these classical results under a unified framework. In fact, we use these results in order to prove our theorem.

\subsection{Notation} Throughout this article unless otherwise mentioned, we will denote by  $k$ 
a finite extension  of $\Q_p$ with ring of integers $\mathcal{O}_k$ and residue field $\kappa$. We will denote by $G_k$ the Galois group, $\Gal(\overline{k}/k)$ of $k$. For a variety $X$ over $k$ and a field extension $L/k$, we will denote by $X_L=X\otimes_k L$ the base change to $L$. 

For an abelian variety $A$ over $k$, we will denote by $\mathcal{A}$ its N\'{e}ron model over $\Spec(\mathcal{O}_k)$. Moreover, for every $n\geq 1$, we will denote by $A[n]$ the group of $n$-torsion points of $A$. The Galois cohomology groups of $k$ will be denoted by $H^i(k,-)$. Finally, all Galois symbol maps considered in this paper will be denoted by $s_n$, but it will always be clear which Galois symbol we are referring to  each time.

\subsection{Acknowledgements} I would like to express my great gratitude to Professor Kazuya Kato for sharing his vision with me that lead to theorem \eqref{mainsemiab}. I am also grateful to Professor Spencer Bloch for his valuable encouragement and for a fruitful discussion that lead to the  finiteness results obtained in \autoref{apps}. Moreover, I would like to thank Professors Bhargav Bhatt, Toshiro Hiranouchi, Mihran Papikian and Dr. Isabel Leal for useful discussions. Finally, I am grateful to the referee whose useful suggestions helped improve significantly the paper. 

\vspace{3pt}
\section{Background material}\label{background} In this section we review some necessary facts about abelian varieties with semistable reduction, Somekawa K-groups and Tate duality.   
\subsection{A filtration on $A[n]$}\label{fildef} 
We start by reviewing the definition of semistable reduction for abelian varieties over $p$-adic fields. 
\begin{defn}\label{semistabledefn}\cite[Expos\'{e} IX, Prop. 3.2, Def. 3.4]{SGA7(I)}
An abelian variety $A$ over a $p$-adic field $k$ is said to have \textit{semistable reduction}, if the connected component, $\mathcal{A}_s^\circ$, of the zero element of the special fiber, $\mathcal{A}_s$, of the N\'{e}ron model $\mathcal{A}$ of $A$ is a semi-abelian variety over the residue field $\kappa$, that is it fits  into a short exact sequence of commutative groups over $\kappa$, \[0\rightarrow \overline{T}\rightarrow \mathcal{A}_s^\circ\rightarrow \overline{C}\rightarrow 0,\] where $\overline{T}$ is a torus and $\overline{C}$ is an abelian variety over $\kappa$.
We say further that $A$ has \textit{split semistable reduction}, if $\overline{T}$ is a split torus. 
\end{defn}
  
From now on we consider an abelian variety $A$ over $k$ with such type of reduction. 
 When $A$ has good reduction, the N\'{e}ron model $\mathcal{A}$  is an abelian scheme over $\Spec(\mathcal{O}_k)$ and for every integer $n\geq 1$ the group scheme $\mathcal{A}[n]$ is finite and flat over $\Spec(\mathcal{O}_k)$. When $A$ has semistable reduction, $\mathcal{A}^\circ[n]$ is in general only flat and quasi-finite over $\Spec(\mathcal{O}_k)$. As such, it has a decomposition as a scheme over $\Spec(\mathcal{O}_k)$,
 \[\mathcal{A}^\circ[n]=\mathcal{A}^\circ[n]^f\bigsqcup\mathcal{A}^\circ[n]',\] where 
$\mathcal{A}^\circ[n]^f$ is a  finite group scheme over $\Spec(\mathcal{O}_k)$ which is a clopen subgroup of $\mathcal{A}^\circ[n]$, and $\mathcal{A}^\circ[n]'\times_{\mathcal{O}_k}\kappa=\emptyset$. (See \cite[Expos\'{e} IX, 2.2.3.1]{SGA7(I)}). Thus, the group schemes $\mathcal{A}^\circ[n]$ and $\mathcal{A}^\circ[n]^f$ have the same special fiber. From now on we will refer to $\mathcal{A}^\circ[n]^f$ as \textit{the finite part} of $\mathcal{A}^\circ[n]$. 

Furthermore, for every  $n\geq 1$, the finite part $\mathcal{A}^\circ[n]^f$ has a \textit{toric subgroup}, $\mathcal{A}^\circ[n]^t\subset \mathcal{A}^\circ[n]^f$ with the property,
\[\mathcal{A}^\circ[n]^t\times_{\mathcal{O}_k}\kappa\simeq\overline{T}[n].\] For details on the definition of the toric part we refer to \cite[Expos\'{e} IX, 2.3]{SGA7(I)} for the coprime-to-$p$ case and to \cite[Expos\'{e} IX, 7.3.1]{SGA7(I)} for the case when $n$ is a $p$-power. 

The above filtration induces a filtration $A[n]\supset A[n]^f\supset A[n]^t\supset 0$ of the $G_k$-module $A[n]$ by setting \[A[n]^\bullet=A[n]\cap\mathcal{A}^0[n]^\bullet(\mathcal{O}_{\overline{k}})\;\;\;\;\text{for }``\bullet"=f,t.\] 
\subsection{The orthogonality theorem}\label{orth} Let $A^\star$ be the dual abelian variety of $A$. Since $A$ and $A^\star$ are isogenous, $A^\star$ has also semistable reduction (\cite[Expos\'{e} IX, 2.2.6, 2.2.7]{SGA7(I)}). Moreover, if we have a short exact sequence of commutative group schemes over $\kappa$,
\[0\rightarrow\overline{T}'\rightarrow\mathcal{A}_s^{\star\circ}\rightarrow\overline{C}'\rightarrow 0,\] then there is an isomorphism $\overline{C}'\simeq \overline{C}^\star$. Moreover, the torii $\overline{T}$ and $\overline{T}'$ have the same rank (\cite[Expos\'{e} IX, Theorem 5.4]{SGA7(I)}).

For every  $n\geq 1$, we consider the Weil pairing,
\[\langle ,\rangle:A[n]\otimes A^\star[n]\rightarrow\mu_n.\] The toric part  of $A[n]$ has then the following description, 
\[A[n]^t=A[n]^f\cap(A^\star[n]^f)^\perp,\] where $(A^\star[n]^f)^\perp$ is the complement of $A^\star[n]^f$ under the Weil pairing. The toric part $A^\star[n]^t$ of the dual has the symmetric property. This property is known as \textit{the orthogonality theorem} (see \cite[Expos\'{e} IX, Theorem 5.2]{SGA7(I)}). 

Next assume that $A$ (and therefore also $A^\star$) has split semistable reduction, with $\overline{T}\simeq\G_{m,\kappa}^{\oplus r}$ for some  $r\geq 1$. Then using the orthogonality theorem we obtain an isomorphism, 
\[A[n]/A[n]^f\simeq\Z/n^{\oplus r}.\] One can see this by computing the Cartier dual, $\Hom(A[n]/A[n]^f,\mu_n)$ which is precisely $A^\star[n]^t$.

\subsection{The component group} For an abelian variety $A$ over $k$ with semistable reduction, the component group $\Phi_A$ is defined as follows,
\[\Phi_A:=\mathcal{A}_s/\mathcal{A}_s^\circ,\] and is therefore a finite abelian group. We consider the \textit{character group} of the maximal torus $\overline{T}$ of $\mathcal{A}_s^\circ$, 
\[X_A:=\Hom_{\overline{\kappa}}(\overline{T}_{\overline{\kappa}},\mathbb{G}_{m,\overline{\kappa}}).\] This is a free abelian group contravariantly associated to $A$. The component group $\Phi_A$ fits into the following short exact sequence, 
\[0\rightarrow X_{A^\star}\rightarrow\Hom(X_A,\Z)\rightarrow\Phi_A\rightarrow 0.\] 
 
 For more details on the component group $\Phi_A$ we refer  to \cite{Conrad/Stein2001}. 
 
\subsection{The Raynaud extension}\label{Raynaud} It remains to understand the quotient $A[n]^f/A[n]^t$. To do this we need to introduce the Raynaud extension, $\mathcal{A}^\sharp$, of $A$. 
Some of the properties of this extension are summarized in the following theorem. 
\begin{theo}\label{Raynaud1}(\cite[Expose IX, section 7]{SGA7(I)}. Let $A$ be an abelian variety over $k$ with split semistable reduction. Then there exists a smooth commutative group scheme $\mathcal{A}^\sharp$ over $\Spec(\mathcal{O}_k)$ and an exact sequence of smooth commutative group schemes over $\Spec(\mathcal{O}_k)$,
\[0\rightarrow\mathcal{T}\rightarrow\mathcal{A}^{\sharp\circ}\rightarrow\mathcal{C}\rightarrow 0\] satisfying the following properties. 
\begin{enumerate}
\item $\mathcal{T}\simeq\G_{m,\mathcal{O}_k}^{\oplus r}$ is a split torus.
\item $\mathcal{C}$ is an abelian scheme.
\item The special fibers of $\mathcal{A}^{\sharp\circ}$ and $\mathcal{A}^\circ$ coincide. 
\item $\Phi_A\simeq\mathcal{A}^\sharp/\mathcal{A}^{\sharp\circ}$.
\item We have an isomorphism of formal groups, $\widehat{\mathcal{A}^{\sharp}}\simeq\widehat{\mathcal{A}}$.
\item For every $n\geq 1$ there are isomorphisms of finite flat group schemes, $\mathcal{A}^{\sharp\circ}[n]\simeq\mathcal{A}^\circ[n]^f$ and $\mathcal{A}^\circ[n]^f/\mathcal{A}^\circ[n]^t\simeq\mathcal{C}[n]$. 
\end{enumerate}
\end{theo} 
\begin{defn}\label{abelianquotient} For an abelian variety $A$ over $k$ with split semistable reduction we will denote by $\mathcal{A}[n]^{\ab}$ the quotient $\mathcal{A}[n]^f/\mathcal{A}[n]^t$ and we will call it the \textit{abelian piece} of $\mathcal{A}[n]$. Using the notation of \autoref{Raynaud1}, we have an isomorphism $\mathcal{A}[n]^{\ab}\simeq\mathcal{C}[n]$. 
\end{defn}
Moreover, we have the following classical $p$-adic uniformization theorem. 
\begin{theo}\label{Raynaud2}(\cite[III.8.1]{Faltings/Chai1990}) In the situation of \autoref{Raynaud1} there is a discrete subgroup $\Gamma\subset\mathcal{A}^{\sharp\circ}(k)$ such that 
\begin{enumerate}
\item $\Gamma$ is a free abelian group of rank $r$ (equal to the rank of $\mathcal{T}$). 
\item For every finite extension $L/k$ there is an isomorphism \[A^\circ(L)\simeq\mathcal{A}^{\sharp\circ}(L)/\Gamma,\] where $A^\circ(L):=\mathcal{A}^{\circ}(\mathcal{O}_L)$.  
\end{enumerate}
\end{theo}
\subsection{A filtration of $G[n]$} Next we want to obtain similar information for a semi-abelian variety $G$ over $k$. 
\begin{defn}\label{filsemiab} Let $G$ be a semi-abelian variety over $k$. Assume that $G$ fits into a short exact sequence of commutative groups over $k$,
\[0\rightarrow T\rightarrow G\rightarrow A\rightarrow 0,\] where $T=\G_m^{\oplus r}$ for some $r\geq 0$ and $A$ is an abelian variety with split semistable reduction. For every integer $n\geq 1$ this induces a short exact sequence of $G_k$-modules,
\[0\rightarrow T[n]\rightarrow G[n]\rightarrow A[n]\rightarrow 0.\]  We define a filtration, $G[n]\supset G[n]^f\supset G[n]^t\supset 0$, of $G[n]$ as follows. Let $G[n]^f$ (resp. $G[n]^t$) be the $G_k$-submodule of $G[n]$ which is such that $G[n]^f/T[n]\simeq A[n]^f$ (resp. $G[n]^t/T[n]\simeq A[n]^t$).
\end{defn}  Note that the \textit{toric part} $G[n]^t$ fits into a short exact sequence, \[0\rightarrow T[n]\rightarrow G[n]^t\rightarrow A[n]^t\rightarrow 0.\] Moreover, it follows directly from the definition that we have an isomorphism \[G[n]^f/G[n]^t\simeq A[n]^f/A[n]^t.\] 

Finally, we want to describe a similar filtration for the Cartier dual, $\Hom(G[n],\mu_n)$.  
\begin{defn}\label{filsemiabdual}
We define $\Hom(G[n],\mu_n)^f$ (resp. $\Hom(G[n],\mu_n)^t$) to be the kernel of 
\[\Hom(G[n],\mu_n)\rightarrow\Hom(G[n]^t,\mu_n)\rightarrow 0,\] (resp. the kernel of $\Hom(G[n],\mu_n)\rightarrow\Hom(G[n]^f,\mu_n)\rightarrow 0$.)
\end{defn}
\vspace{2pt}
\subsection{Mackey functors and Somekawa K-group} Throughout this subsection, $k$ can be any perfect field, not necessarily $p$-adic.
We review some facts about  the Somekawa  $K$-group $K(k;G_1,\ldots,G_r)$ attached to semi-abelian varieties $G_1,\ldots,G_r$ over $k$.  We start with the definition of a Mackey functor. For a more detailed discussion we refer to \cite[p. 13, 14]{Raskind/Spiess2000}. 

A Mackey functor $\mathcal{F}$ is a \textit{presheaf with transfers} on the category of \'{e}tale $k$-schemes. A presheaf with transfers is a usual presheaf $\mathcal{F}$ on 
$(\Spec k)_{et}$ having the following additional property. For every finite morphism $X\stackrel{f}{\longrightarrow} Y$ of \'{e}tale $k$-schemes, in addition to the restriction map $\mathcal{F}(Y)\stackrel{f^\star}{\longrightarrow}\mathcal{F}(X)$, there is also a push-forward map, $\mathcal{F}(X)\stackrel{f_\star}{\longrightarrow} \mathcal{F}(Y)$, which we will call the norm, and denote it by $N_{X/Y}$. Moreover, there is a decomposition $\mathcal{F}(X_1\sqcup X_2)=\mathcal{F}(X_1)\oplus \mathcal{F}(X_2)$. Therefore, $\mathcal{F}$ is fully determined by its value  $\mathcal{F}(L):=\mathcal{F}(\Spec L)$ at every finite extension $L$ over $k$.  
\begin{exmp} Every semi-abelian variety $G$  over $k$ induces a Mackey functor  by assigning to a finite extension $L/k$, $G(L):=\Hom(\Spec L, G)$. For a finite extension $F/L$,  the push-forward  is the norm map on semi-abelian varieties, $N_{F/L}:G(F)\rightarrow G(L)$. 
\end{exmp}
Kahn proved in \cite{Kahn1992} that  the category $MF_k$ of Mackey functors on $(\Spec k)_{et}$ is an \textit{abelian category with a tensor product $\otimes^M$}, whose definition we review here. 
\begin{defn} Let $\mathcal{F}_1,\ldots, \mathcal{F}_r$ be Mackey functors on $(\Spec(k))_{et}$. The Mackey product ${\mathcal{F}_1\otimes^M\cdots\otimes^M \mathcal{F}_r}$ is defined at a finite extension $K$ over $k$ as follows: 
\[(\mathcal{F}_1\otimes^M\cdots\otimes^M \mathcal{F}_r)(K):=\left(\bigoplus_{L/K\text{ finite}} \mathcal{F}_1(L)\otimes\cdots\otimes \mathcal{F}_r(L)\right)/R_1,\] where $R_1$ is the subgroup generated by elements of the form 
\[\label{projectionformula} a_1\otimes\cdots\otimes N_{L'/L}(a_i)\otimes\cdots\otimes a_r-N_{L'/L}(\res_{L'/L}(a_1)\otimes\cdots\otimes a_i\otimes\cdots\otimes \res_{L'/L}(a_r)) \in R_1, \] 
where $L'\supset L\supset K$ is a tower of finite extensions of $k$, $a_i\in \mathcal{F}_i(L')$ for some $i\in\{1,\ldots,r\}$, and $a_j\in \mathcal{F}_j(L)$ for every $j\neq i$.  
\end{defn} 
\begin{notn} From now on we will be  using the standard symbol notation for the generators of $(\mathcal{F}_1\otimes^M\cdots\otimes^M \mathcal{F}_r)(K)$, namely $\{a_1,\ldots,a_r\}_{L/K}$ for $a_i\in \mathcal{F}_i(L)$. 
\end{notn}

\begin{rem}\label{normnotations} We note that the symbol $\{a_1,\ldots,a_r\}_{L/K}\in (\mathcal{F}_1\otimes^M\cdots\otimes^M \mathcal{F}_r)(K)$ is nothing but $N_{L/K}(\{a_1,\ldots,a_r\}_{L/L})$.  The defining relation $R_1$ is classically referred to as \textit{projection formula}. We rewrite it using the symbolic notation:
\begin{equation}\{a_1,\ldots, N_{L/K}(a_i),\ldots,a_r\}_{L/L}=N_{L/K}(\{\res_{L/K}(a_1),\ldots,a_i,\ldots,\res_{L/K}(a_r)\}_{L/L}). 
\end{equation}
\end{rem}
\vspace{1pt}
\subsection*{The Somekawa $K$-group} Let $G_1,\cdots, G_r$ be semi-abelian varieties over $k$. The Somekawa K-group $K(k;G_1,\ldots,G_r)$ attached to $G_1,\ldots,G_r$ is a quotient of the Mackey product, 
\[K(k;G_1,\ldots,G_r)=(G_1\otimes^M\cdots\otimes^M G_r)(k)/R_2,\]
where the relations $R_2$ arise from fucntion fields of curves over $k$. In what follows we will not need the precise definition of $R_2$, therefore we omit it. For more details we refer to the foundational paper \cite{Somekawa1990}). 

An important property of both the Mackey product, $(G_1\otimes^M\cdots\otimes^M G_r)(k)$ and the K-group $K(k;G_1,\ldots,G_r)$ is that they satisfy covariant functoriality, namely if $f:G_i\rightarrow G_i'$ is a homomorphism of semi-abelian varieties for some $i\in\{1,\ldots,r\}$, then it induces a homomorphism, $f_{\star}:K(k;G_1,\ldots,G_i,\ldots,G_r)\rightarrow K(k;G_1,\ldots,G'_i,\ldots,G_r)$.

\subsection*{The generalized Galois Symbol} Let $G_1,\ldots,G_r$ be semi-abelian varieties over $k$ and $n$ an integer invertible in $k$. The Kummer maps, $G_i(L)/n\hookrightarrow H^1(L,G_i[n])$, together with the cup product and the norm map of Galois cohomology (i.e. the corestriction map) induce a generalized Galois symbol, 
\[s_n:K(k;G_1,\ldots,G_r)/n\rightarrow H^r(k,G_1[n]\otimes\cdots\otimes G_r[n]).\] Studying the image of this map is the main goal of this article. 
\subsection{Tate duality}\label{Tateduality} From now on $k$ is a finite extension of $\Q_p$. In this subsection we review how the local Tate duality perfect pairing is defined.  Let $M$ be a finite $G_k$-module annihilated by a positive integer $n\geq 1$. The Tate duality pairing for $H^2$ is a perfect pairing,
\[\langle ,\rangle:H^2(k,M)\times \Hom_{G_k}(M,\mu_n)\rightarrow\Z/n,\] where 
for a class $z\in H^2(k,M)$ and a $G_k$-homomorphism $g:M\rightarrow\mu_n$, $<z,g>$ is defined to be the class of $g_\star(z)\in H^2(k,\mu_n)\simeq\Z/n$. 

It is a simple exercise to show that if $f:M\rightarrow N$ is a homomorphism of $G_k$-modules, then we have a Tate duality compatibility diagram,
\[\xymatrix{
H^2(k,M)\times \Hom_{G_k}(M,\mu_n)\ar[r]^-{\langle ,\rangle_1}\ar@<-10ex>[d]^{f_\star} &\Z/n\\
H^2(k,N) \times  \Hom_{G_k}(N,\mu_n)\ar[r]^-{\langle ,\rangle_2}\ar@<-5ex>[u]^{f^\star} &\Z/n,
}\]
that is, for every $z\in H^2(k,M)$ and $g\in\Hom_{G_k}(N,\mu_n)$ we have an equality \[<f_\star(z),g>_2=<z,f^\star(g)>_1.\]

\vspace{3pt}
\section{Proofs of the main results}\label{proofs}
\subsection{Some preliminary computations}
 In order to simplify the proof of  \autoref{mainsemiab}, we will first prove some important special cases. We start with the following lemma, which can be proved in more generality.  
\begin{lem}\label{special1} Let $G$ be a semi-abelian variety over $k$ that fits into a short exact sequence,
\[0\rightarrow T\rightarrow G\rightarrow A\rightarrow 0,\] where $T$ is a torus and $A$ is an abelian variety over $k$ with potentially good reduction. Let $T'$ be a split torus over $k$. For every $n\geq 1$ the generalized Galois symbol
\[s_n:K(k;G,T')/n\rightarrow H^2(k,G[n]\otimes T'[n])\] is surjective. 
\end{lem}
\begin{proof} Assume that $T'\simeq\G_m^{\oplus s}$ for some $s\geq 0$. We have a decomposition $K(k;G,T')\simeq \oplus^s K(k;G,\G_m)$, as well as $H^2(k,G[n]\otimes T'[n])\simeq \oplus^s H^2(k,G[n]\otimes\mu_n)$, it therefore suffices to prove the lemma when $s=1$. 
 We start by considering the following edge cases. 

\underline{Case 1}: Assume $G=T$ is a torus, which means $A=0$. When $T$ splits over $k$, the result follows by the Merkurjev-Suslin  theorem (\cite{Merkurjev/Suslin1982}). The case of a more general torus follows by \cite[Lemma 3.1]{Yamazaki2009(a)}.

\underline{Case 2}: Assume $G=A$ is an abelian variety with potentially  good reduction. Let $k'$ be a finite extension of $k$ such that $A\otimes_k k'$ has good reduction. We consider the following diagram, 
\[\xymatrix{
& K(k';A,T')/n\ar[r]^{s_n}\ar[d]^{N_{k'/k}} & H^2(k',A[n]\otimes T'[n])\ar[d]^{N_{k'/k}}\\
& K(k;A,T')/n\ar[r]^{s_n} & H^2(k,A[n]\otimes T'[n]),
}\] where the vetical maps are given by the norm. It follows easily by Tate duality (see for example \cite[2.2]{Bloch1981}) that we have an isomorphism, $H^2(k,A[n]\otimes T'[n])\simeq A[n]_{G_k}$.  The norm on the cohomology side therefore coincides with the natural map 
$N_{k'/k}:A[n]_{G_{k'}}\rightarrow A[n]_{G_k},$ from co-invariants of the smaller group to co-invariants of the larger group. In particular, $N_{k'/k}$ is surjective. It suffices therefore to prove surjectivity of the Galois symbol $s_n$ over $k'$. 

 In the more special case when $A$ is the Jacobian variety of a smooth projective curve $C$ over $k$, this result is due to Bloch (\cite{Bloch1981}). One can imitate the proof to make it work for any abelian variety with good reduction. Since this is not essential for the rest of the paper, we postpone this proof until the appendix (\autoref{Bloch}).  

\underline{General Case}: The case when both $T$ and $A$ are non-trivial can be easily reduced to the previous two cases. Namely, we start with the short exact sequence of $G_k$-modules,
\[0\rightarrow T[n]\rightarrow G[n]\rightarrow A[n]\rightarrow 0.\] By applying the right exact functor $\otimes\mu_n$, we obtain the short exact sequence,
\[0\rightarrow T[n]\otimes\mu_n\rightarrow G[n]\otimes\mu_n\rightarrow A[n]\otimes\mu_n\rightarrow 0.\] Notice that the exactness on the left follows by counting the order of each Galois module. We next consider the following commutative diagram 
\[
	\begin{tikzcd} 
	K(k;T,\G_m)/n\ar{r}{f_1}\ar{d}{\alpha} & H^2(k,T[n]\otimes\mu_n)\ar{d}{\gamma}\\
	K(k;G,\G_m)/n\ar{r}{f_2}\ar{d}{\beta} & H^2(k,G[n]\otimes\mu_n)\ar{d}{\delta}\\
	K(k;A,\G_m)/n\ar{r}{f_3}& H^2(k,A[n]\otimes\mu_n).
	\end{tikzcd}
	\] The diagram has the following properties.
	\begin{itemize}
	\item The maps $f_1, f_3$ are surjective by the previous cases.  (The map $f_1$ is in fact an isomorphism by \cite[Lemma 3.1]{Yamazaki2009(a)}.) 
	\item The map $\beta$ is surjective. This follows by the surjectivity of Mackey functors $G/p\rightarrow A/p$. 
	\item The right column is exact. 
	\end{itemize}  By a simple diagram chasing we can now deduce that the map $f_2$ is surjective. 
	
	Although we won't need this information for what follows, we note that in the special case when $A=J(C)$ is the Jacobian variety of a smooth projective curve over $k$ having a $k$-rational point, the map $f_3$ is also an isomorphism. The proof of this result is due to M. Spiess and can be found in the appendix of \cite{Yamazaki2005}.

\end{proof}
\begin{rem} Notice that when $T$ is a split torus and $A$ has good reduction, \autoref{special1} is indeed a special case of \autoref{mainsemiab}. For, the filtration on $G[n]$ takes on the form $G[n]=G[n]^f\supset G[n]^t=\mu_n^{\oplus r}\supset 0$, while  $\Hom(T'[n],\Z/n)\simeq\Z/n^{\oplus s}$ has trivial finite part. According to \autoref{mainsemiab}, a homomorphism $g:G[n]\rightarrow \Hom(T'[n],\mu_n)$ annihilates the Galois symbol if and only if it maps $G[n]^f$ to zero, that is if and only if $g=0$, which is equivalent to  the Galois symbol being surjective. 
\end{rem}
\medskip We need one more special case that we will later apply to the Raynaud extensions of the abelian varieties $A,B$. 

\begin{prop}\label{special2} Let $G_1,G_2$ be semi-abelian varieties over $k$ such that for $i=1,2$ we have short exact sequences
$0\rightarrow T_i\rightarrow G_i\rightarrow A_i\rightarrow 0,$ where $T_i=\G_m^{\oplus r_i}$ for some $r_i\geq 0$ and $A_i$ is an abelian variety with good reduction. Then \autoref{mainsemiab} holds in this case. 
\end{prop}
\begin{proof} In order to see how \autoref{mainsemiab} reads in this case, we need to understand the filtrations on $G_1[n]$ and $\Hom(G_2[n],\mu_n)$. Since the abelian varieties $A_1, A_2$ have good reduction, according to \autoref{filsemiab}, for $i=1,2$ we have equalities, $A_i[n]=A_i[n]^f$, $G_i[n]=G_i[n]^f$, $G_i[n]^t=T_i[n]$. Moreover, for the Cartier dual, $\Hom(G_2[n],\mu_n)$ we have a short exact sequence of $G_k$-modules, 
\[0\rightarrow\Hom(A_2[n],\mu_n)\rightarrow\Hom(G_2[n],\mu_n)\rightarrow\Hom(T_2[n],\mu_n)\rightarrow 0.\] Notice that the exactness on the left follows by counting the orders of each $G_k$-module in the sequence. By \autoref{filsemiabdual}, we have an isomorphism, $\Hom(G_2[n],\mu_n)^f=\Hom(A_2[n],\mu_n)$, with the latter being isomorphic to $A_2^\star[n]$ by the Weil pairing. 
  This is illustrated in the following diagram, where the quotients are indicated on the edges. 

\[
	\begin{tikzcd} G_1[n] \ar[-]{d}{0} & \Hom(G_2[n],\mu_n) \ar[-]{d}{(\Z/n)^{\oplus r_2}} \\ 
	G_1[n]^f  \ar[-]{d}{A_1[n]} & \Hom(G_2[n],\mu_n)^f \ar[-]{d}{A^\star_2[n]} \\ G_1[n]^t \ar[-]{d}{\mu_n^{\oplus r_1}} & \Hom(G_2[n],\mu_n)^t\ar[-]{d}{0}\\
	0 & 0.
	\end{tikzcd}
	\] Let $g:G_1[n]\rightarrow\Hom(G_2[n],\mu_n)$ be a $G_k$-equivariant homomorphism. We need to show that $g$ annihilates the Galois symbol $K(k;G_1,G_2)/n\xrightarrow{s_n} H^2(k,G_1[n]\otimes G_2[n])$ if and only if it satisfies the following properties.\begin{enumerate}
\item $g(G_1[n])\subset \Hom(G_2[n],\mu_n)^f\simeq A_2^\star[n]$.
\item $g(G_1[n]^t)=0$. 
\item The induced homomorphism  $g:A_1[n]\rightarrow A^\star_2[n]$ lifts to a homomorphism of finite flat group schemes over $\Spec(\mathcal{O}_k)$, $\mathcal{A}_1[n]\xrightarrow{\tilde{g}}\mathcal{A}^\star_2[n]$. 
\end{enumerate}
\underline{Step 1:} We first show that property (1) provides a necessary and sufficient condition for the homomorphism $g$ to annihilate symbols of the form $\{a,b\}_{L/k}$ with $a\in G_1(L)$, $b\in T_2(L)$. 

Notice that we have a commutative diagram,
\[
	\begin{tikzcd} K(k;G_1,T_2)/n \ar{r}{s_n}\ar{d} & H^2(k,G_1[n]\otimes T_2[n])\ar{d}{\alpha} \\
	K(k;G_1,G_2)/n\ar{r}{s_n} &  H^2(k,G_1[n]\otimes  G_2[n]).
	\end{tikzcd}
	\]
This in turn gives the following diagram,

\[\label{diagr1}\xymatrix{
&  K(k;G_1,T_2)/n \ar@<-10ex>[d]^-{s_n}\;\;\;\;\;\;\;\;\;\;\;\;\;\;\;\;\;\;\;\;\;\;\;\;\; & \\
& H^2(k,G_1[n]\otimes  T_2[n])\times \Hom_{G_k}(G_1[n],\Z/n^{\oplus r_2})\ar[r]^-{\langle ,\rangle_2}\ar@<-10ex>[d]^{\alpha}&\Z/n\\
& H^2(k,G_1[n]\otimes  G_2[n])\times \Hom_{G_k}(G_1[n],\Hom(G_2[n],\mu_n))\ar[r]^-{\langle ,\rangle_1}\ar@<-10ex>[u]^{\beta}&\Z/n.
}\]
Here the map $\beta$ is induced by postcomposing with the surjection $\Hom(G_2[n],\mu_n)\twoheadrightarrow\Z/n^{\oplus r_2}$. 
	
First assume that $g$ satisfies property (1). This means exactly that $\beta(g)=0$, and hence $<\alpha(x),g>_1=0$, for every $x\in H^2(k,G_1[n]\otimes T_2[n])$. In particular $<\alpha(s_n(z)),g>_1=0$, for every $z\in K(k;G_1,T_2)/n$. The claim then follows by the above commutative diagram and the Tate duality compatibility with respect to $G_k$-homomorphisms (see \autoref{Tateduality}). 

	For the other direction, assume that $G_1[n]\rightarrow\Hom(G_2[n],\mu_n)$ is such that $<s_n(z),g>_1=0$, for every symbol $z\in K(k;G_1,G_2)$. This in particular means that $<\alpha(s_n(z')),g>_1=0$, for every $z'\in K(k;G_1,T_2)$. Using  \autoref{special1}, we get that $<\alpha(w),g>_1=0$, for every $w\in H^2(k,G_1[n]\otimes T_2[n])$.  This in turn gives 
	$<w,\beta(g)>_2=0$, for every $w\in H^2(k,G_1[n]\otimes T_2[n])$. Since the pairing $\langle ,\rangle_2$ is non-degenerate, we get that $\beta(g)=0$, which proves property (2). 
	
\underline{Step 2:} We show that property (2) provides a necessary and sufficient condition for the homomorphism $g$ to annihilate symbols of the form $\{a,b\}_{L/k}$ with $a\in T_1(L)$, $b\in G_2(L)$. 

	This can be proved quite analogously to Step 1, by considering the commutative diagram 
	\[
	\begin{tikzcd} K(k;T_1,G_2)/n \ar{r}{s_n}\ar{d} & H^2(k,T_1[n]\otimes G_2[n])\ar{d}{\gamma} \\
	K(k;G_1,G_2)/n\ar{r}{s_n} &  H^2(k,G_1[n]\otimes  G_2[n]).
	\end{tikzcd}
	\] and the Tate duality compatibility diagram
	\[\xymatrix{
&  K(k;T_1,G_2)/n \ar@<-10ex>[d]^-{s_n}\;\;\;\;\;\;\;\;\;\;\;\;\;\;\;\;\;\;\;\;\;\;\;\;\; & \\
& H^2(k,T_1[n]\otimes  G_2[n])\times \Hom_{G_k}(T_1[n],\Hom(G_2[n],\mu_n))\ar[r]^-{\langle ,\rangle_3}\ar@<-10ex>[d]^{\gamma}&\Z/n\\
& H^2(k,G_1[n]\otimes  G_2[n])\times \Hom_{G_k}(G_1[n],\Hom(G_2[n],\mu_n))\ar[r]^-{\langle ,\rangle_1}\ar@<-10ex>[u]^{\delta}&\Z/n.
}\]
	 Here the map $\delta$ is induced by precomposing with the inclusion $T_1[n]\hookrightarrow G_1[n]$.	
	
\underline{Step 3:} Let  $g:G_1[n]\rightarrow \Hom(G_2[n],\mu_n)$ be a  $G_k$-homomorphism that satisfies properties (1) and (2). We show that property (3) provides a necessary and sufficient condition for $g$ to annihilate the image of $K(k;G_1,G_2)/n$.

We proceed similarly to the previous two steps. We consider the Galois symbol,
\[K(k;A_1,A_2)/n\rightarrow H^2(k,A_1[n]\otimes A_2[n]).\] By Hilbert theorem 90, we have a surjection $G_i(L)\twoheadrightarrow A_i(L)$, for every finite extension $L$ over $k$ and for $i=1,2$. This in turn yields a surjection $K(k;G_1,G_2)/n\twoheadrightarrow K(k;A_1,A_2)/n$. In fact, we have a commutative diagram,
\[
	\begin{tikzcd} K(k;G_1,G_2)/n \ar{r}{s_n}\ar{d}{\pi} & H^2(k,G_1[n]\otimes G_2[n])\ar{d}{\epsilon} \\
	K(k;A_1,A_2)/n\ar{r}{s_n} &  H^2(k,A_1[n]\otimes  A_2[n]).
	\end{tikzcd}
	\]
Moreover, we have a Tate duality compatibility diagram 
\[\xymatrix{
& H^2(k,G_1[n]\otimes  G_2[n])\ar@<-10ex>[d]^{\epsilon}\times \Hom_{G_k}(G_1[n],\Hom(G_2[n],\mu_n))\ar[r]^-{\langle ,\rangle_1}&\Z/n\\
& H^2(k,A_1[n]\otimes  A_2[n])\times \Hom_{G_k}(A_1[n],A_2^\star[n])\ar@<-10ex>[u]^{\eta}\ar[r]^-{\langle ,\rangle_4}&\Z/n,
}\] where the map $\eta$ is induced by precomposing with $G_1[n]\rightarrow A_1[n]$ and postcomposing with $\Hom(A_2[n],\mu_n)\rightarrow\Hom(G_2[n],\mu_n)$. 
Note that a homomorphism $g:G_1[n]\rightarrow\Hom(G_2[n],\mu_n)$ satisfies properties (1) and (2) if and only if it lies in the image of $\eta$. From now on assume that this is indeed the case and we let $g=\eta(g')$ for some $g':A_1[n]\rightarrow A_2^\star[n]$. For an element $z\in K(k;G_1,G_2)$ we have, 
\[<s_n(z),g>_1=<s_n(z),\eta(g')>_1=0\Leftrightarrow<\epsilon(s_n(z)),g'>_4=0\Leftrightarrow<s_n(\pi(z)),g'>_4=0.\] 

Since the map $\pi$ is surjective, we conclude that  the homomorphism $g=\eta(g')$ annihilates the image of the K-group $K(k;G_1,G_2)/n$ in $H^2(k,G_1[n]\otimes G_2[n])$ if and only if  $g':A_1[n]\rightarrow A_2^\star[n]$ annihilates the image of the K-group $K(k;A_1,A_2)/n$ in $H^2(k,A_1[n]\otimes A_2[n])$. The conclusion then follows by the main theorem of \cite{Gazaki2017} (\autoref{goodred} in the current paper), namely that $g'$ lifts to a morphism $\tilde{g}':\mathcal{A}_1[n]\rightarrow\mathcal{A}^\star_1[n]$ of finite flat group schemes over $\Spec(\mathcal{O}_k)$.

\end{proof}
\vspace{2pt}
\subsection{The case of abelian varieties}\label{abeliancase}
In this subsection we are going to prove \autoref{mainsemiab} for abelian varieties $A,B$ over $k$ with split semistable reduction. We want to study the image of the Galois symbol, 
\[K(k;A,B)\xrightarrow{s_n}H^2(k,A[n]\otimes B[n]).\] Since we only care about the image of $s_n$, it is enough to work with the possibly larger group $(A\bigotimes^M B)(k)/n$. We first recall some facts from \cite[page 15]{Raskind/Spiess2000}. 

The Mackey functor $A$ induced by the abelian variety $A$ has a Mackey subfunctor $A^\circ$ which was already considered in \autoref{Raynaud2}. Namely, for a finite extension $L/k$ we define $A^\circ(L)=\mathcal{A}^\circ(\mathcal{O}_L)$ with the obvious norm and restriction maps.  
The quotient $A/A^\circ$ is the Mackey functor defined on finite extensions as
\[(A/A^\circ)(L):=\Phi_{A_L}=\mathcal{A}_{L, s}/\mathcal{A}_{L, s}^\circ.\] Let $\Gamma_A=\Hom(\overline{T}_A,\G_m)$  be the character group of the maximal torus $\overline{T}_A$ of $\mathcal{A}_s^\circ$  and $\Gamma^\star_A$  be the character group of  $\overline{T}_{A^\star}$. When $A$ has split semistable reduction,  the Mackey functor $A/A^\circ$ takes on a more concrete form. Namely, if $L/k$ is a finite extension then $(A/A^\circ)(L)=\Gamma_A/e_{L/k}\Gamma^\star_A$, where $e_{L/k}$ is the ramification index of the extension $L/k$. Moreover, for a tower $L\supset K\supset k$ of finite extensions, the restriction map is given by the formula,
\begin{eqnarray*}
 \res_{L/K}:& \Gamma_A/e_{K/k}\Gamma_A^\star\rightarrow \Gamma_A/e_{L/k}\Gamma_A^\star\\
& x+e_{K/k}\Gamma^\star_A\mapsto e_{L/K}x+e_{L/k}\Gamma^\star_A,
\end{eqnarray*}
while the norm map is the obvious projection, 
\[N_{L/K}:\Gamma_A/e_{L/k}\Gamma_A^\star\twoheadrightarrow \Gamma_A/e_{K/k}\Gamma_A^\star.\] More importantly, all the norm maps are surjective. The following lemma highlights the significance of this fact.

\begin{lem}\label{component1} Let $A$ be an abelian variety over $k$ with split semistable reduction and $\mathcal{F}$ a Mackey functor over $k$ such that for every $n\geq 1$ the multiplication by $n$ map $\mathcal{F}(\overline{k})\xrightarrow{n}\mathcal{F}(\overline{k})$ is surjective. Then the natural map $(A^\circ\bigotimes^M \mathcal{F})(k)/n\rightarrow (A\bigotimes^M \mathcal{F})(k)/n$ is surjective. 
\end{lem}
\begin{proof} We consider the short exact sequence of Mackey functors,
\[0\rightarrow A^\circ\rightarrow A\rightarrow A/A^\circ\rightarrow 0.\] Next, we apply the right exact functor $\bigotimes^M \mathcal{F}$ to obtain an exact sequence,
\[A^\circ\bigotimes^M \mathcal{F}\rightarrow A\bigotimes^M \mathcal{F}\rightarrow (A/A^\circ)\bigotimes^M \mathcal{F}\rightarrow 0.\] We will show that the Mackey functor $(A/A^\circ)\bigotimes^M\mathcal{F}$ is $n$-divisible for every integer $n\geq 1$. 

Let $K/k$ be a finite extension and $\{a,b\}_{L/K}$ be a generator of $((A/A^\circ)\bigotimes^M \mathcal{F})(K)$. Let $b'\in \mathcal{F}(\overline{K})$ be a point such that $nb'=b$. We consider the finite extension $L=K(b')$ over which $b'$ is defined. Since the norm map, $N_{L/K}:(A/A^\circ)(L)\rightarrow (A/A^\circ)(K)$ is surjective, we can write $a=N_{L/K}(a')$ for some $a'\in(A/A^\circ)(L)$. Using the projection formula of $((A/A^\circ)\bigotimes^M \mathcal{F})(K)$, we have an equality,
\[\{a,b\}_{L/K}=\{N_{L/K}(a'),b\}_{L/K}=\{a',\res_{L/K}(b)\}_{L/L}
=\{a',nb'\}_{L/L}=n\{a',b'\}_{L/L},\] which proves the claim.

\end{proof}

We are now ready to prove \autoref{mainsemiab} for abelian varieties with split semi-stable reduction. The theorem in this case reads as follows. 
\begin{theo}\label{mainab2} Let $k$ be a finite extension of the $p$-adic field $\Q_p$, with $p\geq 5$. Let $A,B$ be abelian varieties over $k$ with split semistable reduction. Let $B^\star$ be the dual abelian variety of $B$ and $\mathcal{A},\mathcal{B}^\star$ the N\'{e}ron models of $A,B^\star$ respectively. For a positive integer $n\geq 1$, we consider the Tate duality perfect pairing,
\[\langle ,\rangle:H^2(k,A[n]\otimes B[n])\times\Hom_{G_k}(A[n],B^\star[n])\rightarrow\Z/n.\] Then, the exact annihilator under Tate duality of the image of the Galois symbol, \[\img(K(k;A,B)\xrightarrow{s_n} H^2(k,A[n]\otimes B[n]))\]  consists of those homomorphisms $g:A[n]\rightarrow B^\star[n]$ that have the following properties.
\begin{enumerate}
\item $g$ preserves the filtration \eqref{fil}, i.e. $g(A[n]^\bullet)\subset B^\star[n]^\bullet$, where $``\bullet"=f,t$.
\item The induced map $g:A[n]^f/A[n]^t\rightarrow B^\star[n]^f/B^\star[n]^t$ lifts to a homomorphism $\tilde{g}:\mathcal{A}[n]^{\ab}\rightarrow\mathcal{B}^\star[n]^{\ab}$ of the corresponding finite flat group schemes over $\Spec(\mathcal{O}_k)$. (See \autoref{abelianquotient} for a definition of $\mathcal{A}[n]^{\ab}$). 
\end{enumerate} When the integer $n$ is coprime to $p$, the  assumption $p\geq 5$ can be dropped. 
\end{theo}
\begin{proof} Let $\mathcal{A}^\sharp$, $\mathcal{B}^\sharp$ be the Raynaud extensions of $A$, $B$ respectively (see \autoref{Raynaud}). We will denote by $\mathcal{A}^{\sharp\circ}$, $\mathcal{B}^{\sharp\circ}$ the connected components of the zero element.  Let $A^{\sharp\circ},B^{\sharp\circ}$ be the generic fibers of $\mathcal{A}^{\sharp\circ}$ and $\mathcal{B}^{\sharp\circ}$ respectively. By \autoref{Raynaud1}, $\mathcal{A}^{\sharp\circ}$ and $\mathcal{B}^{\sharp\circ}$ fit into short exact sequences of commutative group schemes over $\Spec(\mathcal{O}_k)$, 
\[0\rightarrow\mathcal{T}_A\rightarrow\mathcal{A}^{\sharp\circ}\rightarrow\mathcal{C}\rightarrow 0\;\;\;\;\;\text{and }\;\;\;\;0\rightarrow\mathcal{T}_B\rightarrow\mathcal{B}^{\sharp\circ}\rightarrow\mathcal{D}\rightarrow 0,\] where $\mathcal{T}_A,\mathcal{T}_B$ are split torii and $\mathcal{C},\mathcal{D}$ are abelian schemes over $\Spec(\mathcal{O}_k)$. The generic fibers, $A^{\sharp\circ},B^{\sharp\circ}$, fit into short exact sequences of commutative groups over $k$, \[0\rightarrow\mathcal{T}_A\otimes_{\mathcal{O}_k}k\rightarrow A^{\sharp\circ}\rightarrow\mathcal{C}\otimes_{\mathcal{O}_k}k\rightarrow 0\;\;\;\;\;\text{and }\;\;\;\;0\rightarrow\mathcal{T}_B\otimes_{\mathcal{O}_k}k\rightarrow B^{\sharp\circ}\rightarrow\mathcal{D}\otimes_{\mathcal{O}_k}k\rightarrow 0.\] In particular they are semi-abelian varieties over $k$ whose maximal abelian quotient is an abelian variety with good reduction. 
Moreover, let $\Gamma_1$, $\Gamma_2$ be the discrete subgroups of $\mathcal{A}^{\sharp\circ}(k)$, $\mathcal{B}^{\sharp\circ} (k)$ as described in \autoref{Raynaud2}.
 We will show a surjection at the level  of Mackey functors, \[(A^{\sharp\circ}\bigotimes^M B^{\sharp\circ})(k)/n\stackrel{\pi}{\longrightarrow} (A\bigotimes^M B)(k)/n\longrightarrow 0.\] 
Note that by part (2) of \autoref{Raynaud2} we have a surjection 
\[(A^{\sharp\circ}\bigotimes^M B^{\sharp\circ})(k)\longrightarrow (A^\circ\bigotimes^M B^\circ)(k)\longrightarrow 0.\]
It suffices therefore to show that for every $n\geq 1$ the map 
\[(A^\circ\bigotimes^M B^\circ)(k)/n\longrightarrow(A\bigotimes^M B)(k)/n\] is surjective. This follows by applying \autoref{component1} twice. This lemma applies because the multiplication by $n$ map $A^\circ(\overline{k})\xrightarrow{n} A^\circ(\overline{k})$ (and similarly for $B^\circ$) is surjective. For, we have a surjection $A^{\sharp\circ}(\overline{k})\twoheadrightarrow A^\circ(\overline{k})$ and multiplication by $n$ is surjective on $A^{\sharp\circ}(\overline{k})$ since $A^{\sharp\circ}$ is a semi-abelian variety. 

\underline{Claim:} The above surjection is compatible with the Galois symbol, namely we have a commutative diagram,
\[
	\begin{tikzcd} (A^{\sharp\circ}\bigotimes^M B^{\sharp\circ})(k)/n \ar{r}{s_n}\ar{d}{\pi} & H^2(k,A^{\sharp\circ}[n]\otimes B^{\sharp\circ}[n])\ar{d}{\iota} \\
	(A\bigotimes^M B)(k)/n\ar{r}{s_n} &  H^2(k,A[n]\otimes  B[n]).
	\end{tikzcd}
	\]
	
First, we note that in the above diagram the map $\iota$ is induced by the inclusions $A^{\sharp\circ}[n]\subset A[n]$, $B^{\sharp\circ}[n]\subset B[n]$. The fact that we do get such inclusions follows from property (6) of \autoref{Raynaud1}. In fact, we have identifications $A^{\sharp\circ}[n]\simeq A[n]^f$ and similarly for $B$. 

In order to prove the claim, it suffices to show that for every $n\geq 1$ and for every finite extension $L/k$, the composition 
 \[
	\begin{tikzcd} A^{\sharp\circ}(L)/n\ar{r}{\delta}\ar{dr} & H^1(L,A[n]^f)\ar{d}{\iota}\\
	& H^1(L,A[n])
	\end{tikzcd}
	\]
factors through $A^\circ(L)/n$, where the map $\delta$ is the connecting homomorhism of the Kummer sequence for $A^{\sharp\circ}$ (and the analogous result for $B$).  The required commutativity follows by \cite[Corollary 7.3]{Faltings/Chai1990}. For, this corollary yields a short exact sequence of $G_k$-modules,
\[0\rightarrow A[n]^f\rightarrow A[n]\rightarrow\Gamma_1/n\rightarrow 0,\] which in turn gives rise to a long exact sequence,
\[\cdots\rightarrow\Gamma_1/n\rightarrow H^1(L,A[n]^f)\xrightarrow{\iota} H^1(L,A[n])\rightarrow\cdots\]

The remaining of the proof is easy, similar to the arguments used in \autoref{special2}. Namely, we consider the Tate duality compatibility diagram, 
\[\xymatrix{
& H^2(k,A[n]^f\otimes  B[n]^f)\ar@<-10ex>[d]^{\iota}\times \Hom_{G_k}(A[n]^f,\Hom(B[n]^f,\mu_n))\ar[r]^-{\langle ,\rangle_1}&\Z/n\\
& H^2(k,A[n]\otimes  B[n])\times \Hom_{G_k}(A[n],B^\star[n])\ar[r]^-{\langle ,\rangle_2}\ar@<-10ex>[u]^{\beta}&\Z/n.
}\]
 Let $g:A[n]\rightarrow B^\star[n]$ be a $G_k$-equivariant homomorphism and $z\in (A\otimes^M B)(k)/n$. Because the map $\pi$ is surjective, we may write $z=\pi(z')$ for some $z'\in(A^{\sharp\circ}\otimes^M B^{\sharp\circ})(k)/n$. We then get the following equivalent statements,
	\[<s_n(z),g>_2=0\Leftrightarrow<s_n(\pi(z'),g>_2=0
	\Leftrightarrow<\iota(s_n(z')),g>_2=0\Leftrightarrow<s_n(z'),\beta(g)>_1=0.\] We conclude that the homomorphism $g$ annihilates the image of $(A\otimes^M B)(k)/n$, if and only if $\beta(g)$ annihilates the image of $(A^{\sharp\circ}\otimes^M B^{\sharp\circ})(k)/n$. To finish the argument, notice that for the semi-abelian varieties $A^{\sharp\circ}$, $B^{\sharp\circ}$ we may apply \autoref{special2}. The proof is then complete, after we observe that $\beta(g)$ satisfying the conclusion of \autoref{special2} is equivalent to $g$ satisfying the conclusion of \autoref{mainab2}.

\end{proof}
\begin{rem} In the previous proof we chose to work with the Mackey product, $(A\bigotimes^M B)(k)/n$ instead of the Somekawa K-group, $K(k;A,B)/n$. The reason is that it is not clear to us whether the analytic maps $A^{\sharp\circ}/\Gamma_1\rightarrow A$ and $B^{\sharp\circ}/\Gamma_2\rightarrow B$ induce a morphism \[K(k;A^{\sharp\circ},B^{\sharp\circ})/n\rightarrow K(k;A,B)/n.\] Before dividing by $n$, the existence of such a map at the level of K-groups is highly unlikely.  When we work $mod \;n$ though it is sometimes the case that the function field relation of the Somekawa K-group is absorbed by the projection formula. (See for example \cite[Lemma 4.2.6]{Raskind/Spiess2000}.) 
\end{rem}
\vspace{2pt}
\subsection{The case of semi-abelian varieties}
In this subsection we want to prove the general case of  \autoref{mainsemiab}. Let $G_1,G_2$ be semi-abelian varieties over $k$ such that for $i=1,2$ there are short exact sequences of commutative groups over $k$,
\[0\rightarrow T_i\rightarrow G_i\rightarrow A_i\rightarrow 0,\] where $T_i=\G_m^{\oplus r_i}$ for some $r_i\geq 0$ and $A_i$ is an abelian variety with split semistable reduction. 


The proof of \autoref{mainsemiab} is quite analogous to the proofs of \autoref{special2} and \autoref{mainab2}, therefore we will only sketch the proof, omitting some details that are left to the interested reader to fill in. 

We first distinguish the following special case.

\begin{lem}\label{special3} Let $T$ be a split torus of rank $r$ and $A$ an abelian variety over $k$ with split semistable reduction. Then \autoref{mainsemiab} holds in this case. More specifically, the image of the generalized Galois symbol, $K(k;T,A)/n\xrightarrow{s_n}H^2(k,T[n]\otimes A[n])$ coincides with the image of the push-forward map
$H^2(k,T[n]\otimes A[n]^f)\xrightarrow{\iota}H^2(k,T[n]\otimes A[n])$. 
\end{lem}
\begin{proof} We can clearly reduce to the case when $r=1$. Let $A^\sharp$ be the special fiber of the Raynaud extension $\mathcal{A}^\sharp$ of $A$ and $A^{\sharp\circ}$ the special fiber of the connected component of zero, $\mathcal{A}^{\sharp\circ}$. We consider the commutative diagram,
\[
	\begin{tikzcd} 
	(\G_m\bigotimes^M A^{\sharp\circ})(k)/n\ar{r}{s_n}\ar{d}{\pi} & H^2(k,\mu_n\otimes A[n]^f)\ar{d}{\iota}\\
	(\G_m\bigotimes^M A)(k)/n\ar{r}{s_n}& H^2(k,\mu_n\otimes A[n]).
	\end{tikzcd}
	\]
By \autoref{special1} the top horizontal map is surjective. Moreover, by \autoref{component1} the map $\pi$ is also surjective and hence the second claim of the lemma follows. The first claim then follows by considering the Tate duality perfect pairing,
\[H^2(k,\mu_n\otimes A[n]^f)\times\Hom_{G_k}(\mu_n,\Hom(A[n]^f,\mu_n))\rightarrow\Z/n.\] We conclude that a $G_k$-homomorphism $g:\mu_n\rightarrow A^\star[n]$ annilates the image of the K-group $K(k;\G_m,A)$, if and only if the induced map 
\[\mu_n\xrightarrow{g} A^\star[n]\rightarrow\Hom(A[n]^f,\mu_n)\] vanishes. This is equivalent to $g(T[n])\subset A^\star[n]^t$.

\end{proof} 
\begin{rem} We note that the above computation has applications to class field theory of curves. Namely, Bloch (\cite{Bloch1981}) constructed a reciprocity map, \[\rho:V(C)\rightarrow \varprojlim\limits_{n}H^2(k, J(C)\otimes\mu_n),\] where $C$ is a smooth projective curve over $k$ with $C(k)\neq\emptyset$ and Jacobian $J(C)$. Moreover, he showed that the inverse limit, $\varprojlim\limits_{n}H^2(k,J(C)\otimes\mu_n)$, is isomorphic to the group $ T(J(C))_{G_k}$ of $G_k$-coinvariants of the Tate module, $T(J(C))$. Here $V(C)$ is a K-group which was proved by Somekawa (\cite[Theorem 2.1]{Somekawa1990}) to be isomorphic to the K-group $K(k;J(C),\G_m)$. As briefly mentioned in the introduction (example \autoref{example2}),  when $C$ has good reduction Bloch proved that the Galois symbol $s_n$ is surjective for every $n\geq 1$, or equivalently that the map $\rho$ has dense image. When $C$ has bad reduction, S. Saito showed that $s_n$ is no longer surjective and he gave a description of the cokernel (\cite{Saito1985}). More recently Yoshida (\cite{Yoshida2003}) revisited this problem and he gave a computation very analogous to \autoref{special3}. 
\end{rem}

We now proceed to a sketch of the proof of \autoref{mainsemiab}. 
\begin{proof} (Sketch)

We are only going to prove the following two special cases. The general case can be reduced to these two by analyzing $G_1$ into its toric and abelian pieces. 

\underline{Special case 1:} When $G_1=T_1$ is a split torus and $G_2$ is semi-abelian. 

In this case \autoref{mainsemiab} amounts to proving that a $G_k$-homomorphism $g:T_1[n]\rightarrow\Hom(G_2[n],\mu_n)$ annilates the image of $K(k;T_1,G_2)/n$, if and only if the induced map 
\[T_1[n]\xrightarrow{g}\Hom(G_2[n],\mu_n)\rightarrow\Hom(G_2[n]^f,\mu_n)\] vanishes. First, $g$ must annihilate the image of $K(k;T_1,T_2)/n$, or equivalently the image of $H^2(k,T_1[n]\otimes T_2[n])\rightarrow H^2(k,T_1[n]\otimes G_2[n])$. This is equivalent to the vanishing of the composition 
$T_1[n]\xrightarrow{g}\Hom(G_2[n],\mu_n)\rightarrow\Hom(T_2[n],\mu_n),$  and hence $g$ annihilates the image of $K(k;T_1,T_2)/n$ if and only if it factors as $T_1[n]\xrightarrow{g}\Hom(A_2[n],\mu_n)\simeq A_2^\star[n]$. The usual Tate compatibility diagram then allows us to reduce to the case when $G_2=A_2$ is an abelian variety with split semistable reduction, in which case the claim follows from \autoref{special3}.   

\underline{Special Case 2:} When $G_1=A_1$ is an abelian variety with split semistable reduction and $G_2$ is semi-abelian. 

First we consider the Tate duality compatibility diagram,
\[\xymatrix{
&H^2(k,A_1[n]\otimes  T_2[n])\ar@<-10ex>[d]^{\alpha}\times \Hom_{G_k}(A_1[n],\Hom(T_2[n],\mu_n))\ar[r]^-{\langle ,\rangle_2} &\Z/n\\
&H^2(k,A_1[n]\otimes  G_2[n]) \times \Hom_{G_k}(A_1[n],\Hom(G_2[n],\mu_n))\ar[r]^-{\langle ,\rangle_1}\ar@<-10ex>[u]^{\beta} &\Z/n.
}
\]
 With the usual argument and using special case 1 we see that a $G_k$-homomorphism $A_1[n]\xrightarrow{g}\Hom(G_2[n],\mu_n)$ annihilates the image of the composition $K(k;A_1,T_2)/n\rightarrow K(k;A_1,G_2)/n\rightarrow H^2(k,A_1[n]\otimes G_2[n])$ if and only if the induced homomorphism \[A_1[n]^f\hookrightarrow A_1[n]\xrightarrow{g}\Hom(G_2[n],\mu_n)\rightarrow \Hom(T_2[n],\mu_n)\]
vanishes, that is $g$ induces a $G_k$-homomorphism $A_1[n]^f\rightarrow\Hom(A_2[n]^f,\mu_n)\simeq A_2^\star[n]$. This together with the commutative diagram,
\[
	\begin{tikzcd} (A_1^{\sharp\circ}\bigotimes^M G_2)(k)/n \ar{r}{s_n}\ar{d}{\pi} & H^2(k,A_1^{\sharp\circ}[n]\otimes G_2[n])\ar{d}{\iota} \\
	(A_1\bigotimes^M G_2)(k)/n\ar{r}{s_n} &  H^2(k,A_1[n]\otimes  G_2[n]),
	\end{tikzcd}
	\] (where the map $\pi$ is surjective by \autoref{component1}), allows us to replace the abelian variety $A_1$ with the good semi-abelian variety $A_1^{\sharp\circ}$ and the semi-abelian variety $G_2$ with its abelian quotient $A_2$. 
	
	The remaining of the proof is of the same flavor. Namely, by considering the corresponding diagram 
	\[
	\begin{tikzcd} (A_1^{\sharp\circ}\bigotimes^M A_2^{\sharp\circ})(k)/n \ar{r}{s_n}\ar{d}{\pi} & H^2(k,A_1^{\sharp\circ}[n]\otimes A_2^{\sharp\circ}[n])\ar{d}{\iota} \\
	(A_1^{\sharp\circ}\bigotimes^M A_2)(k)/n\ar{r}{s_n} &  H^2(k,A_1^{\sharp\circ}[n]\otimes  A_2[n]),
	\end{tikzcd}
	\] we may reduce to the case when $G_1=A_1^{\sharp\circ}$ and $G_2=A_2^{\sharp\circ}$, which follows by \autoref{special2}.

\end{proof}
\vspace{3pt}
\section{Applications}\label{apps}
\subsection{A finiteness result} In this subsection we prove a  finiteness result, which as we shall see in the next subsection, has an important application to zero-cycles. This is the bad reduction analogue of \cite[Corollary 8.3]{Gazaki2017}. 
\begin{prop}\label{finiteness} Let $A_1,A_2$ be abelian varieties over $k$ with split semistable reduction. The map induced by the Galois symbol,
\[\varprojlim\limits_{n\geq 1} K(k;A_1,A_2)/n\stackrel{\lim s_n}{\longrightarrow}\varprojlim\limits_{n\geq 1} H^2(k,A_1[n]\otimes A_2[n])\] has finite image. 
\end{prop}
\begin{proof} Since $\varprojlim\limits_{n\geq 1} H^2(k,A_1[n]\otimes A_2[n])$ is a finitely generated $\hat{\Z}$-module, it suffices to show that there is some integer $N\geq 1$ such that $Ns_n=0$, for every $n\geq 1$. Equivalently, we will show that if $g:A_1[n]\rightarrow A_2^\star[n]$ is a $G_k$-homomorphism, then there is some $N\geq 1$ such that $Ng$ annihilates the image of $s_n$ and $N$ is independent of $n\geq 1$. To do this, we will follow one-by-one the steps of the proof of \autoref{special2}. 

\underline{Claim 1:} There is some $N_1\geq 1$ independent of $n$ such that the homomorphism $N_1g$ has the property  $N_1g(A_1[n]^f)\subset A_2^\star[n]^f$. 

For $i=1,2$ we consider the Raynaud extension $\mathcal{A}_i^\sharp$ of $\mathcal{A}_i$. We set $A_i^\sharp=\mathcal{A}_i^\sharp\times_{\mathcal{O}_k} k$  and we assume we have short exact sequences of commutative groups over $k$, 
\[0\rightarrow T_i\rightarrow A_i^{\sharp\circ}\rightarrow C_i\rightarrow 0,\] where as usual $T_i$ is a split torus  and $C_i$ an abelian variety with good reduction. 
Recall from Step 1 of the proof of \autoref{special2} that the property $N_1g(A_1[n]^f)\subset A_2^\star[n]^f$ is equivalent to $N_1g$ annihilating the image of  \[H^2(k,A_1^{\sharp\circ}[n]\otimes T_2[n])\rightarrow H^2(k,A_1[n]\otimes A_2[n]).\] Note that the group $\varprojlim\limits_{n\geq 1}H^2(k,A_1^{\sharp\circ}[n]\otimes T_2[n])$ is finite. For, we consider the exact sequence \[\varprojlim\limits_{n\geq 1}H^2(k,T_1[n]\otimes T_2[n])\rightarrow\varprojlim\limits_{n\geq 1}H^2(k,A_1^{\sharp\circ}[n]\otimes T_2[n])\rightarrow\varprojlim\limits_{n\geq 1}H^2(k,C_1[n]\otimes T_2[n]).\] The finiteness of $\varprojlim\limits_{n\geq 1}H^2(k,T_1[n]\otimes T_2[n])$ is classical, while $\varprojlim\limits_{n\geq 1}H^2(k,C_1[n]\otimes T_2[n])$  was proved to be finite by S. Saito (\cite[Theorem 1.1 part (2)]{Saito1985}). 
We conclude that we can take $N_1$ to be the order of $\varprojlim\limits_{n\geq 1}H^2(k,A_1^{\sharp\circ}[n]\otimes T_2[n])$. 

\underline{Claim 2:} There is some $N_2\geq 1$ independent of $n$ such that the homomorphism $N_2g$ has the property  $N_2 g(A_1[n]^t)\subset A_2^\star[n]^t$.  

The claim follows analogously to claim 1 by taking $N_2$ to be the order of the finite group $\varprojlim\limits_{n\geq 1}H^2(k,T_1[n]\otimes A_2^{\sharp\circ}[n])$. 

\underline{Claim 3:} There is some $N_3\geq 1$ independent of $n$ such that the induced homomorphism $N_1N_2N_3 g:A_1[n]^{\ab}\rightarrow A_2[n]^{\ab}$ lifts to a homomorphism $N_1N_2N_3\tilde{g}:\mathcal{C}_1\rightarrow\mathcal{C}_2$ of finite flat group schemes over $\Spec(\mathcal{O}_k)$.  

First we note that if $n$ is coprime to $p$, then $N_1 N_2 g$ has the required property. This is because every finite flat group scheme $\mathcal{G}$ over $\Spec(\mathcal{O}_k)$ of order coprime to $p$ is \'{e}tale. For more details we refer to \cite[Proposition 2.6]{Gazaki2017}. We  therefore need to show the existence of a uniform $N_3$ for $G_k$-homomorphisms $N_1 N_2 g:A_1[p^n]^{\ab}\rightarrow A_2^\star[p^n]^{\ab}$. 
The existence of such $N_3$ follows by a result of Bondarko (\cite{Bondarko2006}). In this paper Bondarko shows that the functor $\mathcal{G}\rightarrow \mathcal{G}\times_{\mathcal{O}_k}k$ that sends a finite flat group scheme of order $p^n$ for some $n\geq 1$ to its generic fiber is \textit{weakly full}. In our case this amounts to the existence of an integer $s\geq 0$ that depends only on the ramification index of the field $k$ such that $p^s N_1 N_2g$ lifts to a homomorphism $p^s N_1 N_2\tilde{g}:\mathcal{C}_1\rightarrow\mathcal{C}_2$ of finite flat group schemes over $\Spec(\mathcal{O}_k)$. 

\end{proof}

\subsection{Zero-cycles}
We close the main body of this paper with a corollary about zero-cycles. We remind the reader that for a smooth projective variety $X$ over $k$ having a $k$-rational point, the Chow group of zero-cycles, $CH_0(X)$, has a filtration \[CH_0(X)\supset A_0(X)\supset T(X),\] where $A_0(X):=\ker(\deg)$ is the kernel of the degree map, while $T(X):=\ker(A_0(X)\stackrel{\alb_X}{\longrightarrow}\Alb_X(k))$ is the kernel of the Albanese map.

When $X$ is a finite product of smooth projective curves, $X=C_1\times\cdots\times C_d$, the Albanese kernel $T(X)$ has been related to the Somekawa K-group $K(k;J(C_{i_1}),\cdots,J(C_{i_r}))$  by Raskind and Spiess in \cite[Corollary 2.4.1]{Raskind/Spiess2000}. Moreover, when $X$ is an abelian variety, $T(X)$ has been related to the K-group $K(k;X,\cdots,X)$  by the current author in \cite[Theorem 1.3]{Gazaki2015}. Using these relations, together with \autoref{finiteness} we obtain the following corollary.
\begin{cor} Let $X$ be a smooth projective variety over a finite extension $k$ of $\Q_p$ with $p\geq 5$. We assume that $X$ is  either an abelian variety with split semistable reduction, or a product $C_1\times\cdots\times C_d$ of smooth curves such that for each $i\in\{1,\cdots,d\}$, the curve $C_i$ has a $k$-rational point and the Jacobian variety $J(C_i)$ has split semistable reduction. Then for every prime $l$, the cycle map to \'{e}tale cohomology \[CH_0(X)\rightarrow H^{2d}(X_{et},\Z_l(d)),\] when restricted to the Albanese kernel of $X$ has finite image. 
\end{cor}
\begin{proof} This follows directly from \autoref{finiteness} together with \cite[Proposition 2.4, Lemma 2.5]{Yamazaki2005} for the case of a product of curves and \cite[Section 6]{Gazaki2015} for the case of an abelian variety. 

\end{proof}

\vspace{3pt}
\appendix
\section{Extending the Result of Bloch}
In this appendix we extend \cite[Theorem 2.9]{Bloch1981} to every abelian variety with good reduction. Our proof follows essentially  the same steps as Bloch. 
\begin{theo}\label{Bloch}  Let $A$ be an abelian variety with good reduction over a finite extension $k$ of $\Q_p$. For every $n\geq 1$, the generalized Galois symbol,
\[K(k;A,\G_m)/n\xrightarrow{s_{n}} H^2(k,A[n]\otimes\mu_n)\simeq A[n]_{G_k}\] is surjective. In particular the induced map, $\varprojlim\limits_{n\geq 1}K(k;A,\G_m)/n\rightarrow T(A)_{G_k}$ is surjective. 
\end{theo}
\begin{proof} We fix an integer $n\geq 1$ and we want to prove that the Galois symbol,
\[K(k;A,\G_m)/n\xrightarrow{s_n}A[n]_{G_k}\] is surjective. 
 First we observe that it suffices to show surjectivity after a finite base change. For, if $k'$ is a finite extension of $k$, we have a commutative diagram,
\[
	\begin{tikzcd} 
	K(k';A,\G_m)/n\ar{r}{s_{n}}\ar{d}{N_{k'/k}} & A[n]_{G_{k'}}\ar{d}{N_{k'/k}}\\
	K(k;A,\G_m)/n\ar{r}{s_{n}} & A[n]_{G_{k}}.
	\end{tikzcd}
	\] The norm map $N_{k'/k}:A[n]_{G_{k'}}\rightarrow A[n]_{G}$ is surjective, since it maps coinvariants of the smaller group to coinvariants of the larger group. Therefore, if the top horizontal map is surjective, so is the bottom one. 
	
	From now on, we may therefore assume that $\mu_n\subset k^\times$ and we fix a non-canonical isomorphism $i:\mu_n\xrightarrow{\simeq}\Z/n$. We consider the Weil pairing,
\[\langle ,\rangle_W:A[n]\times A^\star[n]\rightarrow\mu_n.\] If we restrict to the subgroup $A^\star[n](k)$, we obtain a perfect pairing,
\[\langle ,\rangle_W:A[n]_{G_k}\times A^\star[n](k)\rightarrow\mu_n.\] In order to show that the map, $K(k;A,\G_m)/n\xrightarrow{s_{n}} A[n]_{G_k}$ is surjective, it suffices to show that its exact annihilator under the pairing $\langle ,\rangle_W$ is zero. That means, it suffices to show that for every non-trivial $c\in A^\star[n](k)$, there exists some $x\in K(k;A,\G_m)$ with the property $<s_{n}(x), c>_W\neq 0$.
	
Using the non-canonical isomorphism $i:\mu_n\simeq\Z/n$, we get an isomorphism $H^2(k,\mu_n\otimes\mu_n)\simeq H^2(k,\mu_n)\simeq\Z/n$.  Thinking of the element $c\in A^\star[n](k)$ as a $G_k$-homomorphism $A[n]\xrightarrow{c}\mu_n$, we can therefore see that we have an equality,
\[i(<s_{n}(x), c>_W)=c_\star(s_n(x))\in H^2(k,\mu_n\otimes\mu_n),\] where  $c_\star:H^2(k,A[n]\otimes\mu_n)\rightarrow H^2(k,\mu_n\otimes\mu_n)$ is the induced map on cohomology. 

Next we consider the connecting homomorphisms of the Kummer sequences for $A$ and $\G_m$ respectively, $\delta_1:A(k)\rightarrow H^1(k,A[n])$ and $\delta_2:k^{\times }\rightarrow H^1(k,\mu_n)$. Using the definition of the Galois symbol $s_n$, the above isomorphisms and the Merkurjev-Suslin theorem, it is enough to show the following. 

\underline{Goal:} There exists some $a\in A(k)$ and $b\in k^\times$ such that $\{c_{\star}(\delta_1(a)),b\}\neq 0\in K_2^M(k)/n$. 

Since we are allowed to choose the element $b$ freely in $k^\times/k^{\times n}$, and we know that the Hilbert symbol is non-degenerate, the goal is reduced to proving that the composition 
$A(k)\xrightarrow{\delta_1} H^1(k,A[n])\xrightarrow{c_\star}H^1(k,\mu_n)$ is non zero. To do this, we consider a different realization of the group $A^\star[n](k)$. Namely, we consider the connecting homomorphism, 
\[\Hom_{G_k}(A[n],\mu_n)\xrightarrow{\varepsilon}\Ext_k^1(A,\mu_n).\] We see that the class $\varepsilon(c)$ gives rise to an extension $B\in\Ext_k^1(A,\mu_n)$, \[0\rightarrow\mu_n\rightarrow B\xrightarrow{f} A\rightarrow 0.\] We then have a commutative diagram,
\[
	\begin{tikzcd} 
	A(k)\ar{r}{n}\ar{d} & A(k)\ar{d}{=}\ar{r}{\delta_1} & H^1(k,A[n])\ar{d}{c_\star} \ar{r} &\cdots\\
	B(k)\ar{r}{f} & A(k)\ar{r}{\eta} & H^1(k,\mu_n)\ar{r} &\cdots.
	\end{tikzcd}
	\]
We conclude that it suffices to show that the connecting homomorphism  $\eta$ is non trivial, or equivalently that the map $B(k)\xrightarrow{f} A(k)$ is not surjective. This was already proved by Bloch in \cite[page 248]{Bloch1981}. Note that the assumption that $A$ has good reduction is crucial here and the claim is far from being true in the bad reduction case. 

\end{proof}

\vspace{10pt}

\bibliographystyle{amsalpha}

\bibliography{bibfile}

\end{document}